\documentclass[12pt,reqno]{amsart}
\usepackage{amsmath, amsthm, amssymb, stmaryrd}
\usepackage{hyperref}
\usepackage{enumerate}
\usepackage{url}
\usepackage{comment}
\usepackage[dvipsnames]{xcolor}

\let\OLDthebibliography\thebibliography
\renewcommand\thebibliography[1]{
  \OLDthebibliography{#1}
  \setlength{\parskip}{0pt}
  \setlength{\itemsep}{0pt plus 0.3ex}
}

\usepackage{mathtools}

\usepackage{multirow, array}
\usepackage{placeins}

\usepackage{caption} 
\advance \topmargin by -\headheight
\advance \topmargin by -\headsep
     
\evensidemargin \oddsidemargin
\newtheorem{thm}{Theorem}[section]
\newtheorem{lemma}[thm]{Lemma}
\newtheorem{prop}[thm]{Proposition}
\newtheorem{cor}[thm]{Corollary}

\newtheorem{fundamental}[thm]{Fundamental Question}

\theoremstyle{definition}

\theoremstyle{remark}
\newtheorem{remark}[thm]{Remark}

\numberwithin{equation}{section}

\newcommand{\mmod}[1]{{\,\,\mathrm{mod}\,\,#1}}

\newcommand*\wrapletters[1]{\wr@pletters#1\@nil}
\def\wr@pletters#1#2\@nil{#1\allowbreak\if&#2&\else\wr@pletters#2\@nil\fi}

\def\alp{{\alpha}}

\def\del{{\delta}}

\def\lam{{\lambda}}

\def\eps{\varepsilon}

\def\le{\leqslant} \def\ge{\geqslant}

\def\d{{\,{\rm d}}}

\def \bE {\mathbb E}

\def \bN {\mathbb N}
\def \bP {\mathbb P}
\def \bQ {\mathbb Q}
\def \bR {\mathbb R}
\def \bZ {\mathbb Z}
\def \bT {\mathbb T}

\def \cA {\mathcal A}

\def \cE {\mathcal E}

\def \cL {\mathcal L}

\def \cS {\mathcal S}

\begin{document}
\title[Additive energy and the metric Poissonian property]{Additive energy and the metric Poissonian property}
\author[Thomas F. Bloom]{Thomas F. Bloom}
\address{Heilbronn Institute for Mathematical Research, Department of Mathematics, University of Bristol, 
University Walk, Bristol BS8 1TW, United Kingdom}
\email{matfb@bristol.ac.uk}
\author[Sam Chow]{Sam Chow}
\address{Department of Mathematics, University of York, Heslington, York YO10 5DD, United Kingdom}
\email{sam.chow@york.ac.uk}
\author[Ayla Gafni]{Ayla Gafni}
\address{Department of Mathematics, 915 Hylan Building, University of Rochester, Rochester, NY 14627}
\email{agafni@ur.rochester.edu}
\author[Aled Walker]{Aled Walker}
\address{Andrew Wiles Building, University of Oxford, Radcliffe Observatory Quarter, Woodstock Rd, Oxford OX2 6GG, United Kingdom}
\email{walker@maths.ox.ac.uk}
\date{}
\begin{abstract} Let $A$ be a set of natural numbers. Recent work has suggested a strong link between the additive energy of $A$ (the number of solutions to $a_1 + a_2 = a_3 + a_4$ with $a_i\in A$) and the metric Poissonian property, which is a fine-scale equidistribution property for dilates of $A$ modulo $1$. There appears to be reasonable evidence to speculate a sharp Khintchine-type threshold, that is, to speculate that the metric Poissonian property should be completely determined by whether or not a certain sum of additive energies is convergent or divergent. In this article, we primarily address the convergence theory, in other words the extent to which having a low additive energy forces a set to be metric Poissonian.
\end{abstract}
\maketitle

\section{Introduction}
\label{intro} 

The metric Poissonian property is a refined notion of the equidistribution of certain sequences on the unit circle. Initially studied in \cite{RS1998} for its connections to the Berry--Tabor conjecture in quantum mechanics, the property has recently received renewed interest, owing to the fundamental work of Aistleitner--Larcher--Lewko \cite{ALL2017} who, for the first time, revealed an intimate quantitative connection between this property and the combinatorial notion of \emph{additive energy}. We continue this quantitative investigation in one of our main theorems (Theorem \ref{MainThm}). 

In order to state the main result of \cite{ALL2017}, let us formally define the relevant notions. Let $\cA \subseteq \bN$ be an infinite set. For $X \to \infty$ a parameter, put 
\[
A = \cA \cap [1,X], \quad N = \# A, \quad \del = \del(X) = N/X.
\]For $\alpha\in [0,1)$, and $s>0$ some fixed parameter, we define the pair correlation function
\[
F(\alp, s, X, \cA) := N^{-1} \sum_{\substack{a,b \in A, a \ne b \\ \| \alp(a-b) \| \le s/N}} 1.
\]
Here and in the rest of the paper we will use $\|x\|:= \min_{n\in \bZ}|x-n|$. The parameters $s$, $X$, and the underlying set $\cA$ will often be suppressed, with $F(\alpha)$ used to denote the above. 

Considering $\alpha A$ as a subset of $\bR / \bZ$, the average gap length between consecutive elements is $1/N$, at least when $\alpha\notin \mathbb{Q}$. One can view the function $F(\alp)$ therefore as measuring the proportion of pairs $(\alpha a,\alpha b)$ of distinct elements such that the difference $\alpha a - \alpha b \mod 1$ is on the scale of this average gap length. This scale is determined by the parameter $s$. 

The set $\cA$ is said to be \emph{metric Poissonian} if for almost all $\alp$ we have
\[
F(\alp) \to 2s \qquad (N \to \infty)
\]
for all $s > 0$.\\

If $\cA$ is metric Poissonian, then for almost all $\alpha$ the set $\alpha A \mod 1$ mimics a certain statistic which holds almost surely for sequences $(x_1, x_2, \ldots)$ drawn uniformly and independently at random from $[0,1]$. Indeed, if $I \subseteq \bR / \bZ$ is a uniformly chosen random interval of length $2s/N$, let $Y_I$ denote the random variable $\vert I\cap \{x_1,\ldots,x_N\}\vert$. For fixed $k$ and large $N$ one has
\begin{align*}
\bP(Y_I = k) &= \frac{N!}{k!(N-k)!}(2s/N)^k(1-2s/N)^{N-k}\\
&=(1+o_{k,s}(1)) \cdot \exp(-2s) \cdot (2s)^k/k!,
\end{align*} 
so the limiting distribution of $Y_I$ is Poisson with parameter $2s$. In particular its variance tends to $2s$ as $N$ tends to infinity. 

When the points $x_1,x_2,\ldots,x_N$ are deterministic, calculating the variance of $Y_I$ is more or less equivalent to calculating the pair correlations 
\[ N^{-1} \sum\limits_{\substack{x_i,x_j \\ x_i\neq x_j \\ \Vert x_i - x_j\Vert\leqslant t/N}} 1.\] 
If for \textbf{all} $t > 0$ this tends to $2t$ as $N\rightarrow \infty$, then one may conclude that $\operatorname{Var}(Y_I)\rightarrow 2s$. The metric Poissonian property, therefore, is the statement that for almost all $\alpha$ the set $\alpha \cA \mod 1$ matches the second order statistics of a random sequence in $\bR / \bZ$, regarding distribution in short intervals of length $O(1/N)$. The statistics of a random sequence are Poisson, in the large $N$ limit, thus giving some explanation for the name `metric Poissonian'. For more details about this interpretation and connection, we direct the reader to the second section of \cite{Mar2007}.\\

When $\cA$ is the set of squares, or more generally the set of perfect $k^{\mathrm{th}}$ powers for $k\geqslant 2$, the set $\cA$ is metric Poissonian, as was shown by Rudnick and Sarnak in \cite{RS1998}. General lacunary sequences are also metric Poissonian (a result of \cite{RZ1999}). Considerations of continued fractions show that $\cA = \mathbb{N}$ is \emph{not} metric Poissonian, however. As described at the start of the paper, the situation was greatly clarified in \cite{ALL2017}, by the consideration of \emph{additive energy} \[E = E(X) = \# \{ (a,b,c,d) \in A^4: a+b = c+d \}.\] Indeed, the authors of \cite{ALL2017} proved the following theorem.
\begin{thm}
\label{ALLThm}
Assume that $E\ll_\xi N^{3-\xi}$ for some $\xi>0$. Then $\cA$ is metric Poissonian.
\end{thm}
\noindent We leave it to the reader to verify that this theorem implies the aforementioned results for perfect powers and lacunary sequences. We remark here that the trivial upper bound on $E$ is $N^3$, so the required hypothesis here is in some sense rather weak. 

This theorem is neatly complemented by the following result of Bourgain, given in the appendix of \cite{ALL2017}. 

\begin{thm}
Suppose $$\operatorname{limsup}\limits_{N\rightarrow \infty} \frac{E}{N^3}>0.$$ Then $\cA$ is not metric Poissonian. 
\end{thm}

These two theorems suggest a natural route for further investigation, namely to establish whether there is some threshold for the additive energy $E$ which completely determines whether or not $\cA$ is metric Poissonian. Bourgain, in the same appendix of \cite{ALL2017}, gave an example of a set $\cA$ with $E = o(N^3)$ which was not metric Poissonian. This demonstrates that, if such a threshold exists, it is not at order $N^3$. 

The next development was made by the fourth named author in \cite{Wal2017}, answering a question of Nair concerning the primes.
\begin{thm}
The set of prime numbers is not metric Poissonian.
\end{thm}
\noindent In this instance, one has $E \sim N^3(\log N)^{-1}$. The proof of this theorem made use of the divergent part of Khintchine's theorem on Diophantine approximation, and rested on the divergence of the sum $\sum (N\log N)^{-1}$. At this point, we speculate that if $E \sim N^3 \psi(N)$, for some decreasing function $\psi$, then $\cA$ is metric Poissonian if and only if $\sum \psi(N)/N$ converges. 

Lachmann and Technau \cite{Tec} have recently given examples of sets which are not metric Poissonian, and which come within an arbitrary finitely iterated logarithm of this putative threshold. For $r \in \bN$, they showed that there exists a set $\cA$ which is not metric Poissonian, and which satisfies
\[
E \asymp \frac{ N^3} {\log(N) \log_2(N) \cdots \log_r(N)}.
\]
Here $\log_0(N) = N$ and $\log_t(N) = \log (\log_{t-1}(N))$ for $t \in \bN$. Here and throughout, we assume that $N$ is large enough so that any iterated logarithms appearing are well-defined and positive.\\

We now come to the matter of the present paper. Firstly, we offer a large improvement over the results of \cite{ALL2017}, albeit under an extra hypothesis. Indeed, a key feature of Lachmann and Technau's work is that their constructed set is extremely sparse. Assuming instead that $\mathcal{A}$ is quite dense, we show that if $E\ll_\xi N^3(\log N)^{-2-\xi}$ then $\mathcal{A}$ is metric Poissonian. This can be considered a complement to Lachmann and Technau's construction. 

\begin{thm} \label{MainThm}
Assume that 
\begin{equation} \label{conditions}
E \ll_\xi  \frac{N^3} {(\log N)^{2 + \xi}},
\qquad \del \gg_\xi \frac1{ (\log N)^{2+2\xi}},
\end{equation}
for some constant $\xi > 0$.
Then $\cA$ is metric Poissonian.
\end{thm}

If $\mathcal{A}$ were `random', e.g. each element of $\{ 1,2,\ldots, X\}$ chosen independently at random with probability $\delta$, we would expect the energy to be approximately $\delta N^3$. In the above theorem we permit $\delta$ to be smaller, by a power of a logarithm. Thus the theorem holds for all sets of energy $E\ll_\xi N^3 (\log N)^{-2-\xi}$ bar those which are unexpectedly sparse. 

\iffalse
Assuming the Hardy--Littlewood $k$-tuple conjecture, Theorem \ref{MainThm} applies to admissible $k$-tuples of primes, when $k \ge 3$; one uses Brun's sieve to upper bound the energy. Furthermore, our methods show conditionally that the set of twin primes is metric Poissonian: these are primes $p$ such that $p+2$ is also prime.
\fi

For genuinely random sets $\cA$, certain technical obstructions in the proof may be removed, and we obtain the following quantitatively stronger result.

\begin{thm} \label{RandomThm}
Let $C>2$. Let $\cA$ be a random set of natural numbers defined by choosing $x \in \cA$ independently at random with probability 
\[
\begin{cases}
1, &\text{if } x \le 20,\\
(\log x)^{-1} (\log \log x)^{-C}, &\text{if } x > 20.
\end{cases}
\]
Then, with probability 1, the set $\mathcal{A}$ is metric Poissonian. 
\end{thm}

\noindent In this theorem we have conceded a factor of roughly $\log \log N$ from the `Khintchine threshold'. Indeed, we verify in Appendix \ref{EnergyRandom} that
\[
E \asymp \frac{N^3}{(\log N) \cdot (\log \log N)^C}
\]
with probability 1.

In the random setting we can also tackle the divergence side of the problem.
\begin{thm} \label{RandomDivergence}
Let $0 \le C \le 1$. Let $\cA$ be a random set of natural numbers defined by choosing $x \in \cA$ independently at random with probability 
\[
\begin{cases}
1, &\text{if } x \le 20,\\
(\log x)^{-1} (\log \log x)^{-C}, &\text{if } x > 20.
\end{cases}
\]
Then, with probability 1, the set $\mathcal{A}$ is \textbf{not} metric Poissonian. 
\end{thm}
\noindent This complements Theorem \ref{RandomThm}. In a way, Theorem \ref{RandomDivergence} demonstrates that the examples provided in \cite{Tec} behave as one should expect.\\

We return to our fundamental motivating question: is there a universal `additive energy threshold' for the metric Poissonian property? 
\begin{fundamental}
Is it true that if $E(N) \sim N^3 \psi(N)$, for some weakly decreasing function $\psi: \bN \to [0, 1]$, then $\cA$ is metric Poissonian if and only if $\sum \psi(N)/N$ converges?
\end{fundamental}
\noindent Theorem \ref{RandomThm} shows that there are metric Poissonian sets with energy
\[
\Theta \Bigl( \frac{N^3}{(\log N) \cdot (\log \log N)^{2+\eps}} \Bigr),
\]
whilst Lachmann and Technau \cite[Corollary 1]{Tec} construct sets of energy
\[
\Theta \Bigl( \frac{N^3}{(\log N) \cdot(\log \log N)} \Bigr)
\]
that are not metric Poissonian. In a way, we are within a doubly logarithmic factor of answering the question. Moreover, there is now strong evidence that the Khintchine threshold is the correct one, if such a threshold exists.\\

The paper is organised as follows. In section \ref{s2} we describe our method, introducing an idea of Schmidt which lies at the heart of the approach. The subsequent two sections are devoted to proving the key estimates. In section \ref{s5} we demonstrate how these estimates may be improved in the setting of Theorem \ref{RandomThm}. In section \ref{s6}, we use a sandwiching idea we learnt from \cite{ALL2017} to finish the proofs of Theorems \ref{MainThm} and \ref{RandomThm}. Then, in section \ref{s7}, we use a simplified version of Walker's method \cite{Wal2017} to provide an abridged proof of Theorem \ref{RandomDivergence}.
Finally, in Appendix \ref{EnergyRandom}, we compute the additive energy in the random setting.

Regarding notation, throughout $\mu$ denotes Lebesgue measure. For $n \in \bN$ write $[n] = \{ 1,2,\ldots, n\}$. We use Vinogradov and Bachmann--Landau notation, and we put $\bT = [0,1]$. It is convenient to introduce the notation $\widetilde{E}$ for the normalised additive energy $E/N^3$, so that $\widetilde{E}\in (0,1]$. \\

\noindent \textbf{Acknowledgements}

This material is based upon work partially supported by the National Science Foundation under Grant No. DMS-1440140 while three of the authors were in residence at the Mathematical Sciences Research Institute in Berkeley, California, during the Spring 2017 semester. The second and third authors were MSRI Postdocs, and the fourth author a Program Associate, during the Analytic Number Theory Program. The first named author was visiting the Simons Institute for the Theory of Computing, Berkeley, for the Pseudorandomness Program 2017. All four authors are grateful for the working conditions provided by both institutions, at which this work was initiated. The second named author was also supported by EPSRC Programme Grant EP/J018260/1, and the fourth named author by EPSRC grant no. EP/M50659X/1. We thank Ben J. Green and Niclas Technau for helpful conversations. Finally, we thank an anonymous referee for a helpful and detailed report.

\section{Schmidt's trick} \label{s2}

The arguments of Aistleitner, Larcher, and Lewko in \cite{ALL2017} rest upon the bounds of Bondarenko and Seip on GCD-sums \cite{BS2015}, which reach their natural limit once $\widetilde{E}$ is larger than $ \exp((\log N)^{-1/2 + \xi})$. To move above this threshold, we utilise a device of Schmidt from \cite{Sch1960}, which we learnt from Chapter 4 of Harman's book \cite{Har1998}. The idea is that, when counting integer pairs $(n,m)$ such that $\vert \alpha n -m\vert$ is small, one can profitably split into two cases, depending on the size of greatest common divisor of $m$ and $n$. The resulting sum enjoys much better $L^2$ information than was available in \cite{ALL2017}, at a small $L^1$ cost. 

We note at the outset that we will only consider values of $s > 0$ for which $s \asymp 1$. Let
\[
r(n) = \# \{ (a,b) \in A^2: a-b = n, a \ne b \}
\]
and
\[
\cE_n = \bigcup_{\substack{u \le n \\ (u,n) \le T}} \Bigl( \frac{u-s/N}n, \frac{u+s/N}n \Bigr) \mmod 1,
\]
where $T \ge 2$ is a threshold to be specified in due course. Observe that
\[
F(\alp) = \frac{2}{N} \sum_{\substack{n \le X\\\| n \alp \| \le s/N }} r(n).
\]
The idea is to replace $F(\alp)$ by
\[
F^*(\alp) = F^*(\alp, s, X, \cA) = \frac{2}{N}\sum_{\substack{n \le X\\\alp \in \cE_n}} r(n).
\]

We follow the exposition of Harman very closely, though the inclusion of the weight $r(n)$ means we need to redo much of the argument from Chapter 4 of \cite{Har1998}. Another difference is that the threshold $T$, as well as the parameter $2s/N$ of the intervals, is uniform, depending on $N$ rather than $n$. This will actually simplify the proof of several lemmas. 

We first establish two useful lemmas on the size of the sets $\cE_n$. To this end, we introduce the quantity
\[
\Phi(n) =  \# \{u \in [n]: (u,n) \le T \}
\]
and observe that
\begin{equation}\label{intsize}
\mu(\cE_n) = \frac{2s}{N}\cdot \frac{\Phi(n)}{n}.
\end{equation}
The second moment of $\Phi(n)$ will be particularly useful here, but in section 5 we will also have use for the first moment, when considering the random case Theorem \ref{RandomThm}.

\begin{lemma} \label{PsiAvg}
We have
\[
\sum_{n \le X} \Bigl(1 - \frac{\Phi(n)}n \Bigr) \ll \frac{X}{T}.
\]
\end{lemma}

\begin{proof}
We begin by noting that
\[
0 \le n - \Phi(n) = \sum_{\substack{u \le n:\\ (u,n) > T}}1  
=\sum_{\substack{d \mid n: \\ d > T}} \phi\left(\frac{n}{d}\right)\le \sum_{\substack{d \mid n: \\ d > T}} \frac n d,
\]
where $m \mapsto \phi(m)$ is the Euler totient function.
It follows that
\begin{align*}
0\le \sum_{n \le X} \Bigl(1- \frac{\Phi(n)}n \Bigr) &\le \sum_{n \le X} \sum_{\substack{d \mid n: \\ d > T}} \frac1d = \sum_{T < d \le X} \frac1d \sum_{\substack{n \le X: \\ d\mid n}}1 \\
& \le X \sum_{T < d \le X} d^{-2} \ll \frac{X}{T}.
\end{align*}
\end{proof}

\begin{lemma} \label{Psi2}
We have
\[
\sum_{n \le X} \Bigl(1 - \frac{\Phi(n)}n \Bigr)^2 \ll \frac{X \log T}{T^2}.
\]
\end{lemma}

\begin{proof}
Using the same initial manoeuvre as above, we have
\begin{align*}
\sum_{n \le X} \Bigl(1- \frac{\Phi(n)}n \Bigr)^2 &\le \sum_{n \le X} \sum_{\substack{d,e \mid n \\ d,e > T}} \frac1{de} = \sum_{T < d,e \le X} \frac1{de} \sum_{\substack{n \le X: \\ [d,e]\mid n}}1 \\
& \le X \sum_{T < d,e \le X} \frac{(d,e)}{d^2e^2}.
\end{align*}
Putting
\[
g = (d,e), \qquad d = gx, \qquad e = gy
\]
gives
\begin{align*}
\sum_{n \le X} \Bigl(1- \frac{\Phi(n)}n \Bigr)^2 &\le X \sum_{g \le X} g^{-3} \sum_{x,y \ge T/g} (xy)^{-2} \\& \ll X \sum_{g \le T} g^{-1} T^{-2} + X\sum_{g > T} g^{-3} \ll \frac{X \log T}{T^2}.
\end{align*}
\end{proof}

With these estimates done, we move on to the main goal of this section, which is to establish the following $L^1$ estimate. We recall the use of $\widetilde{E}$ to denote the normalised additive energy $E/N^3$. 

\begin{prop} \label{close} We have
\[
\int_\bT |F(\alp) - F^*(\alp)| \d \alp \ll \frac{\sqrt{\log T}}T(\widetilde{E}\delta^{-1})^{1/2}.
\]
\end{prop}
\begin{proof}
Observe that
\[
0 \le F^*(\alp) \le F(\alp),
\]
and that
\[
\int_\bT F(\alp) \d \alp = \frac{2}{N} \sum_{n \le X}  r(n)\frac{2s}N,
\qquad
\int_\bT F^*(\alp) \d \alp = \frac{2}{N} \sum_{n \le X} r(n) \mu(\cE_n).
\]
Therefore
\[
\int_\bT |F(\alp) - F^*(\alp)| \d \alp = \frac{2}{N} \sum_{n \le X} r(n)\Bigl( \frac{2s}N - \mu(\cE_n) \Bigr).
\]
Using \eqref{intsize} it follows that
\begin{align*}
\int_\bT \lvert F(\alp) - F^*(\alp)\rvert \d \alp 
&= \frac{2}{N} \sum_{n \le X} r(n)\frac{2s}{N} \left(1- \frac{\Phi(n)}n \right) \\
&\ll_s 
N^{-2} \sum_{n \le X} r(n) \left(1- \frac{\Phi(n)}{n} \right).
\end{align*}
Finally, Cauchy's inequality and Lemma \ref{Psi2} yield the desired inequality, since
\begin{equation} \label{energy}
E = N^2 + 2 \sum_{n \le X} r(n)^2.
\end{equation}
\end{proof}

\section{An overlap estimate} \label{s3}

In this section, preparing for an $L^2$ argument, we bound the Lebesgue measure of the overlap $\cE_n \cap \cE_m$. For $m, n \in \bN$, define
\[
A(n, m) = \sum_{\substack{u \le n, \text{   }v \le m \\ \frac un = \frac vm \\ (u,n), (v,m) \le T}} 1.
\]
Note by symmetry that $A(n,m) = A(m,n)$. The following lemma is a minor adaptation of Lemma 4.4 of \cite{Har1998}, in order to work with uniform cut-offs. 

\begin{lemma} \label{overlap}
For $n \ge m \ge 1$ we have
\[
\mu(\cE_n \cap \cE_m) \le \frac{4s^2}{N^2} + \frac{2s}N \frac{A(n,m)}n.
\]
\end{lemma}

The constant 4 will be important, so we need to be quite precise here.

\begin{proof}
Considering separately the contributions from pairs of intervals for which $u/n=v/m$, it is easy to see that
\[
\mu(\cE_n \cap \cE_m) \le B_1 + B_2,
\]
where
\[
B_1 = \sum_{ \substack{u \le n, v\le m \\ (u,n), (v,m) \le T \\ u/n = v/m}} \frac {2s}{Nn} = \frac{2s}N \frac{A(n,m)}n
\]
and
\[
B_2 = \sum_{ \substack{u \le n, v\le m \\ (u,n), (v,m) \le T \\ u/n \ne v/m}} \mu \Bigl( \Bigl(\frac{u - s/N}n, \frac{u+s/N}n \Bigr) \cap
 \Bigl(\frac{v - s/N}m, \frac{v+s/N}m \Bigr) \Bigr).
\]
It therefore remains to prove that $B_2 \le 4s^2/N^2$.

If $u/n \ne v/m$ and the intervals $\Bigl(\frac{u - s/N}n, \frac{u+s/N}n \Bigr)$ and $\Bigl(\frac{v - s/N}m, \frac{v+s/N}m \Bigr)$ intersect then for some integer $h$ we have
\[
0 < |h| < (n+m)\frac sN, \qquad \frac u n - \frac v m = \frac h {mn}.
\]
The size of the overlap is bounded above by
\[
\min \Bigl(\frac{2s}{Nn}, \frac s N \Bigl( \frac1n + \frac1m \Bigr) - \frac {|h|}{mn} \Bigr).
\]
The number of positive integer solutions to $um - vn = h$ with $u \le n$ and $v \le m$ is bounded above by
\[
\begin{cases}
(m,n) &\text{if } (m,n) \mid h \\
0, &\text{else.}
\end{cases}
\]
With $y = (n+m)s/N$, we now have
\begin{align*}
B_2 &\le \sum_{\substack{0 < |h| < y \\ (m,n) \mid h}} (m,n) \min \Bigl(\frac{2s}{Nn}, \frac s N \Bigl( \frac1n + \frac1m \Bigr) - \frac {|h|}{mn} \Bigr)
\\ &= \sum_{0 < |t| < y/(m,n)} (m,n) \min \Bigl(\frac{2s}{Nn}, \frac s N \Bigl( \frac1n + \frac1m \Bigr) - \frac {(m,n)|t|}{mn} \Bigr).
\end{align*}

The summand is non-increasing in $|t|$, so
\begin{align*}
B_2 &\le 2 \int_0^{y/(m,n)} (m,n)  \min \Bigl(\frac{2s}{Nn}, \frac s N \Bigl( \frac1n + \frac1m \Bigr) - \frac {(m,n)t}{mn} \Bigr) \d t \\
& = 2 \int_0^y \min \Bigl(\frac{2s}{Nn}, \frac s N \Bigl( \frac1n + \frac1m \Bigr) - \frac {h}{mn} \Bigr) \d h.
\end{align*}
Now
\begin{align*}
B_2 & \le 2 \int_0^{(n-m)s/N} \frac {2s}{Nn} \d h + 2 \int_{(n-m)s/N}^{(n+m)s/N} \Bigl( \frac s N  \Bigl( \frac1n + \frac1m \Bigr) - \frac h {mn} \Bigr) \d h 
\\ &= \frac{4s^2}{N^2} \frac{n-m}n + \frac{4s^2}{N^2} m \Bigl(\frac1n + \frac1m\Bigr) - \frac1{mn} \frac{s^2}{N^2} ((n+m)^2 - (n-m)^2),
\end{align*}
and finally we obtain
\[
B_2 \le \frac{4s^2}{N^2} \Bigl(1 - \frac m n + \frac m n + 1 - 1 \Bigr) = \frac{4s^2}{N^2}.
\]
\end{proof}

We shall also need to bound the average overlap. 

\begin{lemma}\label{lem-averageoverlap}
We have
\[
\sum_{n \le X} \sum_{m \le n} \frac{A(m,n)}n \ll X \log T.\]
\end{lemma}
\begin{proof}
We upper bound $\sum \limits_{m\leqslant n} A(m,n) $ by neglecting one of the greatest common divisor constraints, i.e.$$\sum \limits_{m\leqslant n} A(m,n) \leqslant \sum\limits_{\substack{1\leqslant m,u,v\leqslant n\\ \frac{u}{m} = \frac{v}{n}\\ (v,n)\leqslant T}} 1.$$ Now, for fixed $v$, there are unique coprime positive integers $a$ and $b$ such that $\frac{v}{n} = \frac{a}{b}$. The denominator $b$ is a divisor of $n$, and $a\leqslant b$. Further, with such an $a$ and $b$ fixed, the number of solutions to $$\frac{a}{b}= \frac{u}{m}, \qquad 1\leqslant u,m\leqslant n$$ is exactly $\frac{n}{b}$.

We now consider how many times a particular divisor $b$ of $n$ can occur, as $v$ ranges up to $n$. Indeed, since $b= \frac{n}{(v,n)}$, we see that $b$ occurs exactly the number of times that $(v,n) = \frac{n}{b}$. If $\frac{n}{b}\leqslant T$, then this is at most $b$ times; if $\frac{n}{b}> T$, then the greatest common divisor condition $(v,n)\leqslant T$ precludes this from ever occurring. 

Therefore
\[
\sum\limits_{m \le n} A(m,n) \leqslant \sum\limits_{\substack{b\vert n\\ \frac{n}{b} \leqslant T}} b \frac{n}{b} 
 = \sum\limits_{\substack{c\vert n\\ c\leqslant T}} n.\]
It follows that
\[\sum_{n\le X}\sum_{m\le n}\frac{A(m,n)}{n}
\le \sum_{n\le X}\sum_{\substack{c\mid n\\ c\le T}} 1
\le \sum_{c\le T}\frac{X}{c}
\ll X\log T\]
as claimed.
\end{proof}

\section{The variance estimate}

Like so much work on metric properties, we aim to show a result for almost all $\alpha$ by bounding the variance ($\alpha$ considered as a uniform random variable on $\bT$). Rather than working directly with $F$, however, the objective of this section is instead to establish the following bound on the variance of $F^*$.

\begin{prop} \label{variance}
The variance of $F^*$ satisfies
\[
\mathrm{Var}(F^*) =\int_\mathbb{T} \lvert F^*(\alpha) - \mathbb{E} F^*\rvert^2\d \alpha \ll  \frac{\sqrt{\log T}}T (\widetilde{E} \delta^{-1})^{1/2} + \widetilde{E} T \log T.
\]
\end{prop}

\begin{proof}
For brevity, we introduce the temporary notation 
\[
\rho = \frac{\sqrt{\log T}}T (\widetilde{E}  \delta^{-1})^{1/2}.
\]
We will eventually choose $T$ in such a way that $N^{-1} \le \rho \le 1$. 

We begin by replacing $\bE F^*$ with a simpler expression. Indeed, by Proposition \ref{close} and the fact that $F^*(\alp) \le F(\alp)$, we have
\[
\bE F^* = \bE F + O(\rho).
\]
Since
\begin{equation*}
\bE F = \int_\bT F(\alp) \d \alp  = \frac{1}{N}\sum_{n \in \bZ\backslash\{0\}} r(n)\frac{2s}N = 2s (1-1/N),
\end{equation*}
\noindent we must have
\begin{equation} \label{eqn expectation estimation}
\bE F^* = 2s + O(\rho).
\end{equation}
A short calculation then tells us that 
\begin{equation*}
\textrm{Var}(F^*) = \int_\bT (F^*(\alp) - 2s)^2 \d \alp + O(\rho^2).
\end{equation*}

Expanding the main term gives
\begin{align*}
\frac{4}{N^2}\sum_{1\le n,m \le X} r(n) r(m) \mu(\cE_n \cap \cE_m) - \frac{8s}{N}\sum_{1\le n \le X} r(n) \mu(\cE_n) + 4s^2.
\end{align*}
By \eqref{intsize} and Lemma \ref{overlap} this is at most
\begin{align}  \notag
&\frac{16s^2} {N^4} \sum_{n,m \le X} r(n)r(m) 
+ \frac{ 16s }{N^{3}} \sum_{m \le n \le X}  r(n) r(m) \frac{A(n,m)}n 
\\ \label{AledCalc0} &  \qquad -\frac{16s^2 }{N^2} \sum_{n \le X} r(n) \frac{\Phi(n)}n + 4s^2.
\end{align}

Now, as 
\[
\sum_{1\le n \le X} r(n) = \frac{N^2}2 + O(N),
\]
we have
\[
\frac{16s^2}{N^4} \sum_{1\le n,m \le X} r(n)r(m) + 4s^2 = \frac{16s^2}{N^2} \sum_{1\le n \le X} r(n) + O(N^{-1}).
\]
Substituting this into \eqref{AledCalc0} and bounding $\frac{1}{n}$ by $\frac{1}{\sqrt{mn}}$ when $m \le n$ yields
\begin{equation} \label{AledCalc}
\textrm{Var}(F^*) \ll S_1 + S_2 + N^{-1}  + \rho^2,
\end{equation}
where
\[
S_1 =\frac{1}{N^2} \sum_{n \le X} r(n) \left(1- \frac{\Phi(n)}{n}\right)
\]
and
\[
S_2 = N^{-3} \sum_{m, n \le X} r(m) r(n) \frac{A(m,n)}{\sqrt{mn}}.
\]
Observe that the quantity $S_1$ was studied in the proof of Proposition \ref{close}, and in view of that calculation we have $S_1 \ll \rho$.

For $S_2$, a more sensitive treatment of $A(m,n)$ than Lemma \ref{lem-averageoverlap} will be required. We recall that $A(m,n)$ counts solutions $(u,v) \in [m] \times [n]$ to
\[
un = vm, \qquad (u,m), (v,n) \le T.
\]
The solutions are
\[
u = \lam \frac m{(m,n)}, \qquad v = \lam \frac n {(m,n)}
\]
for $1 \le \lam \le (m,n)$, subject to the further restriction given by $T$. As
\[
(u,m) = \frac m{(m,n)} (\lam,m,n), \qquad (v,n) = \frac n{(m,n)} (\lam,m,n),
\]
we have
\[
A(m,n) \le \sum_{\substack{\lam \le (m,n): \\ (\lam, m, n) \le \frac{ (m,n)}{\max\{m,n\}} T}}1.
\]

From this we extract two key pieces of information: that
\begin{equation}\label{first gcd bound}
A(m,n) \le (m,n),
\end{equation}
and that if $A(m,n) \ne 0$ then
\begin{equation}\label{second gcd bound}
\frac{m}{(m,n)}, \frac{n}{(m,n)} \le T.
\end{equation}
Just using \eqref{first gcd bound} along with the gcd sum bounds of \cite{BS2015}, which entirely removes the influence of $T$, we recover an analogous estimate to \cite[Lemma 3]{ALL2017} (see Remark \ref{BigRemark} below). However, if we also use \eqref{second gcd bound} and Cauchy--Schwarz, we obtain
\begin{align*}
S_2 & \le \frac{1}{N^3}\sum_{\substack{m,n\le X\\  \frac{m}{(m,n)}, \frac{n}{(m,n)} \le T}}r(m)r(n)\frac{(m,n)}{\sqrt{mn}} \le \frac{1}{N^3} \sum_{g \le X} \Bigl(\sum_{y \le T} \frac{r(gy)}{\sqrt{y}} \Bigr)^2 \\
&\le \frac1{N^3} \sum_{g \le X} \Bigl(\sum_{y \le T} \frac1y\Bigr) \cdot \sum_{y \le T} r(gy)^2 \ll \frac{\log T}{N^3} \sum_{n \in \bN} r(n)^2 \sum_{\substack{y \mid n \\ y \le T}}1 \ll \widetilde{E} T \log T.
\end{align*}
In view of \eqref{AledCalc}, this completes the proof of the proposition.
\end{proof}

\begin{remark} \label{BigRemark} There is an alternative approach to bounding $S_2$, based on the general theory of gcd sums of the form
\[
\sum_{k,\ell = 1}^M \frac{ (n_k,n_\ell)^{2\alp}}{(n_kn_\ell)^\alp}.
\]
At $\alp = 1/2$ we have already mentioned the essentially optimal bounds by Bondarenko and Seip \cite{BS2015, BS2017}, which were employed in \cite[Lemma 3]{ALL2017}. We can pass to $\alp = 1$ by Rankin's trick, using \eqref{second gcd bound}; at this exponent there is G\'al's \cite{Gal} prize-winning upper bound $O(M (\log \log M)^2)$. By \cite[Lemma 4]{ABS} and the discussion immediately following it, one can attach real weights to this at little cost, the idea being to apply this with weights $r(n)$. The sharpest version is a recent breakthrough by Lewko and Radziwi\l{}\l{} \cite[Theorem 2]{LR}, by which we obtain the $L^2$ estimate
\[
\sum_{k,\ell = 1}^M c_k c_{\ell} \frac{ (n_k,n_\ell)^2 }{n_kn_\ell} \ll  (\log \log M)^2 \sum_k c_k^2.
\]
%The earlier result \cite[Lemma 4]{ABS} has $(\log\log M)^4$ in place of $(\log \log M)^2$, which would also have been adequate for our purposes. 
There are at most $M=N^2$ values of $n$ for which $r(n) \ne 0$, so this approach ultimately yields
\[
\mathrm{Var}(F^*) \ll \frac{\sqrt{\log T}}T (\widetilde{E} \delta^{-1})^{1/2} + T \widetilde{E} (\log \log N)^2.
\]
As it happens, for our eventual choice of $T$ there is %very 
little quantitative difference between the two approaches, and certainly either is sufficient for Theorem \ref{MainThm}.
\end{remark}

\section{Improved estimates for the random case} \label{s5}

In this section we consider the setting of Theorem \ref{RandomThm}, and revisit the variance estimate from the preceding section. By a standard application of large deviation inequalities, we can assume that the representation function $r(n)$ is essentially constant. This leads to an improvement. 

We begin with some easy bounds on $N$ and $r(n)$.

\begin{lemma}
\label{Chernoff lemma}
Let $C>2$, and let 
\[
\psi(x) =
\begin{cases}
1, &\text{if } x \le 20,\\
(\log x)^{-1} (\log \log x)^{-C}, &\text{if } x > 20.
\end{cases}
\]
Let $\cA$ be as defined in the statement of Theorem \ref{RandomThm}, i.e. the random set of natural numbers defined by choosing $x\in \cA$ independently at random with probability $\psi(x)$. Let $\varepsilon>0$, and fix an integer $X$ large enough in terms of $\varepsilon$. Consider the following properties: 
\begin{enumerate}
\item $N$ satisfies 
\[
1-\varepsilon \leqslant \frac N{X(\log X)^{-1}(\log\log X)^{-C}} \leqslant 1+\varepsilon.
\]
\item For all positive integers $n\leqslant X$, we have
\[
r(n) \leqslant (1+\varepsilon)X(\log X)^{-2}(\log\log X)^{-2C}.
\]
\end{enumerate}
Then there is a constant $c_\varepsilon>0$ such that property (1) holds with probability at least $1-O(\operatorname{exp}(-c_\varepsilon X(\log X)^{-1}(\log\log X)^{-C})$ and property (2) holds with probability at least $1-O(X\operatorname{exp}(-c_\varepsilon X(\log X)^{-5}))$.
\end{lemma}
This is a quantitative version of a lemma which appears in Bourgain's appendix to \cite{ALL2017}, and is no doubt obvious to experts. Yet to keep the paper as self-contained as possible, we feel it is appropriate to provide the full details, particularly for part (2). 
\begin{proof}
For each $x\in \mathbb{N}$, let $\xi_x$ denote the Bernoulli random variable such that $$\mathbb{P}(\xi_x = 1) = \psi(x).$$ Then $ N = \sum\limits_{x\leqslant X}\xi_x$ is a random variable with expectation $\sum\limits_{x\leqslant X}\psi(x)$, which is asymptotic to $X(\log X)^{-1}(\log\log X)^{-C}$. Since the $\xi_x$ are independent by assumption, one may settle part (1) immediately by applying the Chernoff bound, for instance by applying Corollary A.1.14 of \cite{AS2008}. \\

For part (2), we first consider each $n\leqslant X-1$ separately. Indeed, for $x$ in the range $1\leqslant x\leqslant X-n$, let $\omega_{x,n}$ denote the random variable $\omega_{x,n} = \xi_x\xi_{n+x}$. We have $$r(n) =\sum\limits_{x\leqslant X-n } \omega_{x,n},$$ which is a random variable with expectation $\sum\limits_{x\leqslant X-n}\psi(x)\psi(n+x)$. Suppose first that $n\geqslant X - (1+\varepsilon)X(\log X)^{-2}(\log \log X)^{-2C}$. Then by the trivial triangle inequality bound we have $r(n)\leqslant (1+\varepsilon)X(\log X)^{-2}(\log \log X)^{-2C}$.

It remains to consider the case $n\leqslant X - (1+\varepsilon)X(\log X)^{-2}(\log \log X)^{-2C}$. We wish to apply concentration of measure results for sums of independent random variables, and the family of random variables $\{\omega_{x,n}:x\leqslant X-n\}$ is very close to being independent. Methods of splitting this family into groups of genuinely independent random variables are alluded to in the discussion in the appendix of \cite{ALL2017}. Here we describe an extremely coarse decomposition which nonetheless is strong enough for our purposes. \\

Split $[X-n]$ into two sets, $\cS_{0}$ and $\cS_{1}$, constructed as follows. If $n \le X/3$, let $x \in \cS_j$ if $ \lfloor x/n \rfloor \equiv j \mmod 2$. If $n > X/3$, instead let 
\[
\cS_0 = \Bigl\{ x \in \bN: x \le \frac{X-n}2 \Bigr\}, \qquad \cS_1 = \Bigl\{ x \in \bN: \frac{X-n}2 < x \le X-n \Bigr\}.
\]
For each $j \in \{0,1\}$ the family $\{\omega_{x,n}: x\in \cS_j\}$ is independent, as no two indices differ by $n$. Applying the union bound and then Corollary A.1.14 of \cite{AS2008} once more, we have 
\begin{align}
\notag
&\mathbb{P} \Bigl(r(n)\geqslant (1+\varepsilon)\sum\limits_{x\leqslant X-n}\psi(x)\psi(n+x) \Bigr)\nonumber\\
\notag &\ll \bP \Bigl(\sum\limits_{x\in \cS_0}\omega_{x,n}\geqslant (1+\varepsilon)\sum\limits_{x\in \cS_0} \psi(x)\psi(n+x) \Bigr)
\\ \label{splitbound}&\qquad + \bP \Bigl(\sum\limits_{x\in \cS_1}\omega_{x,n}\geqslant (1+\varepsilon)\sum\limits_{x\in \cS_1} \psi(x)\psi(n+x) \Bigr)\nonumber\\
&\ll\operatorname{exp}\Bigl(-c_\varepsilon\sum\limits_{x\in \cS_0}\psi(x)\psi(n+x) \Bigr)+\operatorname{exp} \Bigl(-c_\varepsilon \sum\limits_{x \in \cS_1}\psi(x)\psi(n+x) \Bigr).
\end{align}

It remains to estimate these final quantities. Note that by construction we have $\min(\vert \cS_0\vert,\vert \cS_1\vert)\geqslant (X-n)/4$. So, since $\psi(x)$ is weakly decreasing, we deduce from \eqref{splitbound} that
\begin{align*}
\mathbb{P} \Bigl(r(n)\geqslant (1+\varepsilon)\sum\limits_{x\leqslant X-n}\psi(x)\psi(n+x) \Bigr)
&\ll \operatorname{exp} \Bigl(-c_\varepsilon\sum\limits_{X-\frac{X-n}{4}< x\leqslant X}\psi(x)^2 \Bigr)
\\& \ll \operatorname{\exp} (-c_\varepsilon X(\log X)^{-5} )
\end{align*}
by a simple calculation, reducing the quantity $c_\varepsilon$ as necessary. 

By a similar monotonicity principle, $\bP(r(n)\geqslant (1+\varepsilon)\sum\limits_{x\leqslant X-n} \psi(x)\psi(n+x))$ is at least $\bP(r(n)\geqslant (1+\varepsilon)\sum\limits_{x\leqslant X}\psi(x)^2).$ Since 
\[
\sum\limits_{x\leqslant X}\psi(x)^2\sim X(\log X)^{-2}(\log \log X)^{-2C},
\]
and $X$ large enough in terms of $\varepsilon$, we conclude that 
\[
r(n)\leqslant (1+\varepsilon)X(\log X)^{-2}(\log \log X)^{-2C}
\]
with probability greater than $1-O(\operatorname{exp}(-c_\varepsilon X(\log X)^{-5}))$.

All these calculations were done for a single $n$. But applying a crude union bound over all $n \leqslant X-1$, and noting that $r(X) = 0$, part (2) is proved.
\end{proof}

\begin{cor}
\label{flatrep}
With probability $1$ there exists some $X_0$ such that properties (1) and (2) hold for all $X\geqslant X_0$. 
\end{cor}
\begin{proof}
The sums 
\[
\sum\limits_{X\geqslant 1}\operatorname{exp}(-c_\varepsilon X(\log X)^{-1}(\log \log X)^{-C})
\]
and
\[
\sum\limits_{X\geqslant 1}X\operatorname{exp}(-c_\varepsilon X(\log X)^{-5})
\]
are both convergent. So by the first Borel--Cantelli lemma, the corollary follows.
\end{proof}

For the rest of this paper, whenever we consider the random case we restrict to the probability $1$ event from Corollary \ref{flatrep}. \\

To continue this section, we use the above work to get an improved version of Proposition \ref{variance}, in the random case. First we improve on Proposition \ref{close}.

\begin{prop}
\label{close random} In the setting of Theorem \ref{RandomThm}, we have
\[\int_{\mathbb{T}}\vert F(\alpha) - F^*(\alpha)\vert \, \d\alpha \ll T^{-1}.\]
\end{prop} 
\begin{proof}
Following Proposition \ref{close}, we establish that \[\int_{\mathbb{T}}\vert F(\alpha) - F^*(\alpha)\vert \, \d\alpha = 2\sum\limits_{n\leqslant X} \frac{r(n)}{N} \frac{2s}{N}\left(1-\frac{\Phi(n)}{n}\right)\ll X^{-1} \sum\limits_{n\leqslant X}\left(1-\frac{\Phi(n)}{n}\right).\] The final inequality follows from Corollary \ref{flatrep}. Then we apply Lemma \ref{PsiAvg}, and the proposition follows.
\end{proof}
\noindent We now go on to improve Proposition \ref{variance}.
\begin{prop}
\label{variance random}
Let $\cA$ be as in Theorem \ref{RandomThm}. Then the variance of $F^*$ satisfies
\[
\mathrm{Var}(F^*)\ll T^{-1} +(\log X)^{-1}(\log \log X)^{-C}\log T.
\]
\end{prop}

\begin{proof}
We will eventually choose $T$ in such a way that $N^{-1}\leqslant T^{-1}\leqslant 1$. Following the proof of Proposition \ref{variance}, we thus derive 
\[ 
\textrm{Var}(F^*) \ll S_1 + S_3 + N^{-1} + T^{-2},
\] 
where 
\[
S_3 = N^{-3}\sum\limits_{m\leqslant n\leqslant X}r(n)r(m) \frac{A(m,n)}{n}.
\]

Observe that $S_1$ is the same expression considered in Proposition \ref{close random}, and so $S_1\ll T^{-1}$. It remains to bound $S_3$, for which we already have all the necessary lemmas in place. Indeed, by Corollary \ref{flatrep}, and by Lemma \ref{lem-averageoverlap}, we have \begin{align*}
S_3 &\ll X^{-1}(\log X)^{-1}(\log \log X)^{-C}\sum\limits_{m\leqslant n\leqslant X}\frac{A(m,n)}{n}\\& \ll (\log X)^{-1}(\log \log X)^{-C}\log T.
\end{align*}
Combining everything together yields the proposition.
\end{proof}

\section{Sandwiching, concentration, and Borel--Cantelli} \label{s6}

Here we conclude the proofs of Theorems \ref{MainThm} and \ref{RandomThm}. We begin by proving the following assertion, which we will see implies that $\cA$ is metric Poissonian.

\begin{prop} \label{suffices}
Let $\eps, s > 0$. Let $\mathcal{A}$ satisfy either the conditions of Theorem \ref{MainThm} or \ref{RandomThm}. Then there is a full measure set $\Omega_{\eps, s}$ such that if $\alp \in \Omega_{\eps,s}$ then there exists $X_0 = X_0(\alp,\eps,s)$ such that if $X \ge X_0$ then $|F(\alp) - 2s| \le \eps$.
\end{prop}
First, let us describe the argument in the deterministic setting of Theorem \ref{MainThm}.
\begin{proof}
Define the sequence $N_j = \lfloor 2^{j^{1-\eta}} \rfloor$, where $\eta > 0$ is small in terms of $\xi$. Note that $N_{j+1}/N_j \to 1$. For $j \in \bN$, define $X_j \in \bN$ to be minimal such that $N_j = |\cA \cap [X_j]|$. Let $X$ be large in terms of $\eps$ and $s$, and suppose $X_j \le X < X_{j+1}$. We begin with the inequalities noted in \cite{ALL2017}, namely
\[
N_j F(\alp, s \frac{N_j}{N_{j+1}}, X_j) \le NF(\alp, s, X) \le N_{j+1} F(\alp, s \frac{N_{j+1}}{N_j}, X_{j+1}),
\]
 which follow immediately from the definitions. Note that
\[
 s \frac{N_j}{N_{j+1}}, s \frac{N_{j+1}}{N_j} \asymp 1, \qquad  s \frac{N_j}{N_{j+1}}, s \frac{N_{j+1}}{N_j} \to s.
\]

Motivated by Proposition \ref{variance}, we choose $T=(\widetilde{E}  \delta)^{-1/4}+1\geq 2$. By Proposition \ref{close}, we thus have
\[
\int_{\mathbb{T}} \lvert F(\alpha) - F^*(\alpha)\rvert \d \alpha \ll (\log T)^{1/2}\widetilde{E} ^{3/4}\delta^{-1/4}.
\]
At this stage we invoke our hypotheses \eqref{conditions}: with $y := \widetilde{E}^{-1} \gg (\log N)^{2+\xi}$, we have
\[
(\log T)^{1/2} \widetilde{E}^{3/4}
\ll \frac{ \sqrt {\log y} + \sqrt{\log \log N}}{y^{3/4}}.
\]
The right hand side is decreasing in $y$ for $y \ge 100$, so by our initial assumptions \eqref{conditions} we obtain
\begin{equation} \label{eqn L1 est explicit}
\int_{\mathbb{T}} \lvert F(\alpha) - F^*(\alpha)\rvert \d \alpha \ll \frac1{(\log N)^{1+\xi/5}}.
\end{equation}

The same calculation also shows that $N^{-1}\leqslant \rho \leqslant 1$, which was assumed in the proof of Proposition \ref{variance}. Applying Proposition \ref{variance}, therefore, we have
\begin{equation}
\label{eqn L2 est explicit}
\textrm{Var}(F^*)\ll \widetilde{E}^{3/4}\delta^{-1/4} \log T \ll \frac{1}{(\log N)^{1+\xi/5}}.
\end{equation}
The $L^1$ estimate (\ref{eqn L1 est explicit}) implies that if $s \asymp 1$ and $X$ is large then
\[
|F(\alp, s, X) - F^*(\alp, s, X)| \le \eps/4 \qquad (\alp \notin E_{X,s}),
\]
for some exceptional set $E_{X,s}$ with 
\[
\mu(E_{X,s}) \ll_\eps \frac1{(\log N)^{1+\xi/5}}.
\]
In particular
\[
\sum_{j=1}^\infty \mu(E_{X_j, s \frac{N_j}{N_{j+1}}}) \ll \sum_{j=1}^\infty j^{(\eta-1)(1+\xi/5)} < \infty,
\] provided that $\eta$ is chosen small enough depending on $\xi$. Now the Borel--Cantelli lemma tells us that for almost all $\alp \in \bT$ we have
\[
\Bigl |F(\alp, s \frac{N_j}{N_{j+1}}, X_j) - F^*(\alp, s \frac{N_j}{N_{j+1}}, X_j) \Bigr| \le \eps/4 \qquad (j \ge j_1(\alp, \eps)),
\]
and similarly
\[
\Bigl |F(\alp, s \frac{N_{j+1}}{N_j}, X_{j+1}) - F^*(\alp, s \frac{N_{j+1}}{N_j}, X_{j+1}) \Bigr| \le \eps/4 \qquad (j \ge j_2(\alp, \eps)).
\]

Recalling (\ref{eqn expectation estimation}), the variance estimate (\ref{eqn L2 est explicit}) on $F^*$ combined with Chebyshev's inequality implies that 
\[
\mu \Bigl( \Bigl \{ \alp \in \bT: \Bigl|F^*(\alp, s \frac{N_j}{N_{j+1}}, X_j) - 2s \frac{N_j}{N_{j+1}} \Bigr| > \eps/4 \Bigr\} \Bigr) \ll_\varepsilon (\log N_j)^{-1-\xi/5}.
\]We again apply the Borel--Cantelli lemma, and find that for almost all $\alp \in \bT$ we have
\[
\Bigl|F^*(\alp, s \frac{N_j}{N_{j+1}}, X_j) - 2s \frac{N_j}{N_{j+1}} \Bigr|  \le \eps/4 \qquad (j \ge j_3(\alp, \eps)),
\]
and similarly
\[
\Bigl|F^*(\alp, s \frac{N_{j+1}}{N_j}, X_{j+1}) - 2s \frac{N_{j+1}}{N_j} \Bigr|  \le \eps/4 \qquad (j \ge j_4(\alp, \eps)).
\]

Now the triangle inequality gives, for almost all $\alp \in \bT$ and all $j \ge j_5(\alp, \eps)$,
\[
F(\alp, s \frac{N_j}{N_{j+1}}, X_j) \ge 2s \frac{N_j}{N_{j+1}}  - \eps/2
\]
and
\[
F(\alp, s \frac{N_{j+1}}{N_j}, X_{j+1}) \le 2s \frac{N_{j+1}}{N_j} + \eps/2.
\]
Since 
\[
F(\alp, s \frac{N_j}{N_{j+1}}, X_j) \le \frac N{N_j} F(\alp, s, X)
\]
and
\[
F(\alp, s \frac{N_{j+1}}{N_j}, X_{j+1}) \ge \frac N{N_{j+1}} F(\alp, s, X),
\]
we obtain
\[
\frac{N_j}N \Bigl(2s \frac{N_j}{N_{j+1}} - \eps/2   \Bigr)  \le F(\alp, s, X)  \le \frac{N_{j+1}}N \Bigl(2s \frac{N_{j+1}}{N_j} + \eps/2\Bigr).
\]
As $j$ is large and $N_{j+1}/N_j \to 1$, we conclude as claimed that
\[
|F(\alp, s, X) - 2s| \le \eps.
\]
\end{proof}

Now let us prove Proposition \ref{suffices} in the random setting of Theorem \ref{RandomThm}.
\begin{proof}
For some small $\eta>0$, ultimately depending on $\varepsilon$ and $s$, define the sequence $X_j = \lfloor 2^{j\eta}\rfloor$. Let $X$ be large in terms of $\eps$ and $s$, and suppose that $X_j \le X < X_{j+1}$. We note again the sandwiching inequalities \[
N_j F(\alp, s \frac{N_j}{N_{j+1}}, X_j) \le NF(\alp, s, X) \le N_{j+1} F(\alp, s \frac{N_{j+1}}{N_j}, X_{j+1}).
\]
Observe that \[
 s \frac{N_j}{N_{j+1}}, s \frac{N_{j+1}}{N_j} \asymp 1,
\] and that for $j$ large enough in terms of $\eta$ we have 
\[
 \frac{N_j}{N_{j+1}}\geqslant 2^{-2\eta}, \qquad \frac{N_{j+1}}{N_{j}}\leqslant 2^{2\eta}.
\]

Motivated by Proposition \ref{variance random}, we choose 
\[
T = \log ^2 X.
\]
By Proposition \ref{close random}, we find that if $s\asymp 1$ and $X$ is large in terms of $\eta$ then \[
|F(\alp, s, X) - F^*(\alp, s, X)| \le \eta/4 \qquad (\alp \notin E_{X,s}),
\]
for some exceptional set $E_{X,s}$ with 
\[
\mu(E_{X,s}) \ll_\eta (\log X)^{-2}.\] This means that \[\sum\limits_{j=1}^\infty \mu(E_{X_j,s\frac{N_j}{N_{j+1}}}) \ll_\eta\sum\limits_{j=1}^\infty j^{-2}<\infty.\] We apply the first Borel--Cantelli lemma as in the previous proof, concluding that for almost all $\alpha\in \mathbb{T}$ we have 
\[
\Bigl |F(\alp, s \frac{N_j}{N_{j+1}}, X_j) - F^*(\alp, s \frac{N_j}{N_{j+1}}, X_j) \Bigr| \le \eta/4 \qquad (j \ge j_1(\alp, \eta)),
\]
and similarly
\[
\Bigl |F(\alp, s \frac{N_{j+1}}{N_j}, X_{j+1}) - F^*(\alp, s \frac{N_{j+1}}{N_j}, X_{j+1}) \Bigr| \le \eta/4 \qquad (j \ge j_2(\alp, \eta)).
\]

By Proposition \ref{variance random}, the variance of $F^*$ is bounded above by a constant times \[ (\log X)^{-2} + (\log X)^{-1}(\log\log X)^{-C+1}.\] We absorb the first term, and have by Chebychev's inequality that 
\begin{align*}
&\mu \Bigl(  \Bigl \{ \alp \in \bT: \Bigl|F^*(\alp, s \frac{N_j}{N_{j+1}}, X_j) - 2s \frac{N_j}{N_{j+1}} \Bigr| > \eta/ 4 \Bigr \} \Bigr) \\
&\qquad \ll_\eta (\log X_j)^{-1}(\log\log X_j)^{-C+1}.
\end{align*}
We again apply the Borel--Cantelli lemma, using the fact that \[\sum\limits_{j=1}^\infty (\log X_j)^{-1}(\log\log X_j)^{-C+1} \ll \:\eta^{-1}\sum\limits_{j=1}^\infty j^{-1}(\log \eta j)^{-C+1}<\infty\] since $C>2$. Thus we again find that for almost all $\alp \in \bT$ we have
\[
\Bigl|F^*(\alp, s \frac{N_j}{N_{j+1}}, X_j) - 2s \frac{N_j}{N_{j+1}}\Bigr|  \le \eta/4 \qquad (j \ge j_3(\alp, \eta)),
\]
and similarly
\[
\Bigl|F^*(\alp, s \frac{N_{j+1}}{N_j}, X_{j+1}) - 2s \frac{N_{j+1}}{N_j} \Bigr|  \le \eta/4 \qquad (j \ge j_4(\alp, \eta)).\]

The rest of the proof proceeds as in the deterministic case, reaching the expression 
\[
\frac{N_j}N \Bigl(2s \frac{N_j}{N_{j+1}} - \eta/2\Bigr) \le F(\alp, s, X) \le \frac{N_{j+1}}N \Bigl(2s \frac{N_{j+1}}{N_j} + \eta/2\Bigr).
\]
It is not true in this setting that $N_j/N_{j+1}\rightarrow 1$, but by our earlier observations we may establish \[2^{-2\eta}\Bigl(2s \cdot 2^{-2\eta}-\frac{\eta}{2}\Bigr)\leqslant F(\alpha,s,X)\leqslant 2^{2\eta}\Bigl(2s \cdot 2^{2\eta}+\frac{\eta}{2}\Bigr),\] 
and therefore \[ F(\alpha,s,X) = 2s+ O_s(\eta).\] Choosing $\eta$ small enough, this error is at most $\varepsilon$, thus proving the proposition. 
\end{proof}
To complete the proofs of Theorems \ref{MainThm} and \ref{RandomThm}, it remains to deduce the metric Poissonian property from Proposition \ref{suffices}. Apply Proposition \ref{suffices} to all $s \in \bQ_{> 0}$ simultaneously, with an exceptional set of measure zero. Now let $s$ lie in a short interval $(s_1,s_2)$, where $s_1, s_2 \in \bQ$, and the desired conclusion follows.

\iffalse
Indeed, with reference to the notation of Proposition \ref{suffices}, put
\[
\Omega = \bigcap_{\eps, s \in \bQ_{>0}} \Omega_{\eps,s}.
\]
This is a full measure set. Now let $s,\eps > 0$, and $\alpha\in \Omega$. Choose $s_1, s_2 \in \bQ_{>0}$ such that
\[
s_1 < s < s_2, \qquad |s-s_1|, |s-s_2| \le \eps/4.
\]
For $X$ large we have
\[
|F(\alp, s_i, X) - 2s_i| \le \eps/2 \qquad (i=1,2),
\]
so by the triangle inequality
\[
|F(\alp, s_i, X) - 2s| \le \eps \qquad(i=1,2).
\]
As $F$ is monotonic in $s$ we now have
\[
|F(\alp, s, X) - 2s| \le \eps.
\]
Since $\eps > 0$ was arbitrary, we conclude that $F(\alp) \to 2s$. So $\cA$ is metric Poissonian.
\fi

\section{The random divergence theory} \label{s7}

In this section we prove Theorem \ref{RandomDivergence}. We keep this brief, as the proof is straightforward by combining section \ref{s5} with the crux of \cite{Wal2017}. 

By Khintchine's theorem \cite[Theorem 2.2]{Har1998}, we know that there is a full measure set $\Omega$ such that for all $\alpha \in \Omega$ there exist arbitrarily large positive integers $M$ such that 
\begin{equation}
\label{Khintchine}
\| M\alp \| < (M \cL(M))^{-1},
\end{equation} where 
\[
\cL(M) = \prod_{i \le 5} \log_i(M).
\] We remind the reader that here $\log_0(M) = M$ and $\log_t(M) = \log (\log_{t-1}(M))$ for $t \in \bN$.  

Now let $\cA$ be as in Theorem \ref{RandomDivergence}. With negligible alteration, the proofs of Lemma \ref{Chernoff lemma} and Corollary \ref{flatrep} demonstrate that the associated quantities $\del$ and $r(n)$ (for $1 \le n \le X/3$ with $X$ large) have the expected order of magnitude. We conclude specifically that, with probability $1$, there exists $X_0$ such that the following hold.
\begin{enumerate}
\item If $X \ge X_0$ then
\[
\frac12 (\log X)^{-1} (\log \log X)^{-C} \le \del \le 2 (\log X)^{-1} (\log \log X)^{-C}.
\]
\item If $X \ge X_0$ and $1 \le n \le X/3$ then $r(n) \ge \frac18 \del N$.
\end{enumerate}

Now let $\alp$ be a member of $\Omega$, and suppose that there exists an $X_0$ satisfying (1) and (2) above. Let $s > 0$ be constant, let $M$ be a large positive integer satisfying (\ref{Khintchine}), and put 
\[
N = \lfloor M \log_3 (M) \rfloor, \qquad K = \log M(\log \log M)^C \log_4 M.
\]
Note that property (1) implies that $X\gg M(\log M) (\log\log M)^C(\log_3 M)$, and in particular that $KM \leqslant X/3$ (for large enough $M$).

By considering those $n$ of the form $n = kM$ ($1 \le k \le K$), and using property (2), we have that
\[
F(N) \ge N^{-1} \sum_{\substack{n \le X/3 \\ \| n \alp \| < s/N}} r(n) \gg \del K \gg \log_4(M).
\] In particular, with probability $1$, for all $\alpha \in \Omega$ there exist arbitrarily large $N$ such that $F(N) > 3s$. Hence, with probability $1$, $\cA$ is not metric Poissonian.

\appendix
\section{The energy in the random setting} \label{EnergyRandom}

In this appendix we include the computation of the additive energy in the random setting of Theorem \ref{RandomThm}. Recall \eqref{energy}. By Cauchy--Schwarz and Corollary \ref{flatrep}, we have $\log N \asymp \log X$ and
\[
E \ge \sum_{n \le X} r(n)^2 \ge X^{-1} \Bigl(\sum_{n \le X} r(n) \Bigr)^2 \asymp \frac{N^4}X \asymp \frac{N^3}{(\log N) \cdot (\log \log N)^C}
\]
almost surely. Moreover, Corollary \ref{flatrep} also yields
\[
E \ll N^2 +\sum_{n \le X} (N^2/X)^2 \ll \frac{N^4}X \asymp \frac{N^3}{(\log N) \cdot (\log \log N)^C}
\]
almost surely. We thus conclude that
\[
E \asymp \frac{N^3}{(\log N) \cdot (\log \log N)^C}
\]
with probability 1.

\providecommand{\bysame}{\leavevmode\hbox to3em{\hrulefill}\thinspace}

\end{document}